\documentclass[psamsfonts,fceqn,leqno]{amsart}
  \usepackage{amsthm}
  \usepackage{amsmath}
  \usepackage{amssymb}
  \usepackage{mathrsfs}
  
  \newtheorem{theorem}{Theorem}[section]
  \newtheorem{corollary}[theorem]{Corollary}
  \newtheorem{proposition}[theorem]{Proposition}
  \newtheorem{lemma}[theorem]{Lemma}
  \newtheorem{question}[theorem]{Question}
  \newtheorem{problem}[theorem]{Problem}
  \theoremstyle{definition}

  \numberwithin{equation}{section}

  \title[Hereditarily Normal Wijsman Hyperspaces Are Metrizable]
  {Hereditarily Normal Wijsman Hyperspaces\\ Are Metrizable$^\ast$}
  \author[J. Cao]{Jiling Cao}
  \address{School of Computing and Mathematical Sciences,
  Auckland University of Technology, Private Bag 92006, Auckland
  1142, New Zealand}
  \email{jiling.cao@aut.ac.nz}
  \author[H.~J.~K. Junnila]{Heikki J. K. Junnila}
  \address{Department of Mathematics and Statistics, The
  University of Helsinki, P.~O. Box 68, FI-00014,  Helsinki,
  Finland}
  \email{heikki.junnila@helsinki.fi}
  \thanks{\hspace{-1.66em} 2010 \emph{Mathematics Subject
  Classification.}
  Primary 54E35; Secondary 54B20, 54D15.}
  \thanks{\noindent $^\ast$This paper was initially and
  partially written when the first author was in a Research
  and Study Leave from July to December 2009, and visited
  the second author in August 2009. The paper was eventually
  completed when the two authors met and discussed at the
  International Conference on Topology and the Related Fields,
  held at Nanjing, China, 22-25 September 2012. The two authors
  would like to thank the School of Computing and Mathematical
  Sciences at the Auckland University of Technology, and the
  Department of Mathematics and Statistics at the University
  of Helsinki for their financial supports.}

  \date{}

  \dedicatory{Dedicated to Professor Mitrofan Choban
  and Professor Stoyan Nedev\\
  for their 70 birthday.}

\begin{document}

  \begin{abstract}
  In this paper, we study normality and metrizability of Wijsman hyperspaces.
  We show that every hereditarily normal Wijsman hyperspace is metrizable.
  This provides a partially answer to a problem of Di Maio and Meccariello
  in 1998.
  \end{abstract}

  \maketitle

  \section{Introduction}

  Throughout this paper, $2^X$ denotes the family of all nonempty closed
  subsets of a given topological space $X$. For a metric space $(X,d)$,
  let $d(x,A)=\inf\{d(x,a): a \in A\}$
  denote the distance between a point $x \in X$ and a nonempty subset
  $A$ of $(X,d)$, and
  $S_d(A,\varepsilon)=\{ x \in X: d(x,A)<\varepsilon\}$.
  A net $\{A_\alpha: \alpha \in D\}$ in
  $2^X$ is said to be \emph{Wijsman convergent to} some $A$ in $2^X$ if $d(x,A_\alpha) \to d(x, A)$ for every $x\in X$. The Wijsman topology
  on $2^X$ induced by $d$, denoted by $\tau_{w(d)}$, is the weakest
  topology such that for every $x\in X$, the distance functional $d(x,
  \cdot)$ is continuous. To see the structure of this topology, for
  any $E \subseteq X$, let
  $E^- =\{A \in 2^X: A \cap E\ne \emptyset \}$.
  It can be seen easily that the Wijsman topology on $2^X$ induced by
  $d$ has the family
  \[
  \left\{U^-: U \mbox{ is open in } X \right\}\cup
  \left\{\{A \in 2^X: d(x,A) >\varepsilon \} : x\in X,
  \varepsilon >0 \right\}
  \]
  as a subbase, refer to \cite{beer:93}. Moreover, for a finite
  subset $E\subseteq X$ and $\varepsilon >0$, let
  \[
  {\mathcal N}_{A,E,\varepsilon}=\left\{B\in 2^X: |d(x,A)-d(x,B)|
  <\varepsilon \mbox{ for all } x\in E\right\}.
  \]
  Then for any $A \in 2^X$, the collection
  \[
  \{{\mathcal N}_{A,E,\varepsilon}: E\subseteq X \mbox{ is
  finite and } \varepsilon >0\}
  \]
  forms a neighborhood base of $A$ in $\tau_{w(d)}$. This type of
  convergence was first introduced
  by Wijsman in \cite{wijsman:66} for sequences of closed convex sets
  in Euclidean space $\mathbb R^n$, when he studied optimum properties
  of the sequential probability ratio test. It was
  \cite{lechicki-levi:87} where Wijsman convergence was considered
  in the general framework of a metric space, and the metrizability
  of the Wijsman topology of a separable metric space was
  established. Since then, there has been a considerable effort
  to explore various topological properties of Wijsman hyperspaces.
  For example, Beer \cite{beer:91} and Costantini
  \cite{costantini:95} studied Polishness of Wijsman hyperspaces,
  Cao and Tomita \cite{cao-tomita:10} as well as Zsilinszky
  \cite{zsilinszky:07} investigated Baireness of Wijsman hyperspaces,
  Cao and Junnila \cite{cao-junnila:10} studied Amsterdam properties
  of Wijsman spaces. However, Wijsman hyperspaces are far to be
  completely understood, and still there are many problems concerning
  fundamental properties of these objects unsolved. This motivates
  the authors to continue their study of Wijsman hyperspaces in the
  present paper.

  \medskip
  Note that all Wijsman topologies are Tychonoff, since they are
  weak topologies. In a more recent paper, Cao, Junnila and Moors
  \cite{cao-junnila-moors:12} showed that Wijsman hyperspaces are
  universal Tychonoff spaces in the sense that every Tychonoff
  space is embeddable as a closed subspace in the Wijsman hyperspace
  of a complete metric space which is locally $\mathbb R$. Thus,
  one of the fundamental problems is to determine when a Wijsman
  hyperspace is normal. The problem was first mentioned
  by Di Maio and Meccariello in \cite{diMM:98}, where it was
  asked whether the normality of a Wijsman hyperspace is equivalent
  to its metrizability. A partial solution to this problem, which
  asserts that the answer is ``yes" when the base space of a Wijsman
  hyperspace is a normed linear space, was recently observed by Hol\'{a}
  and Novotn\'{y} in \cite{hola-novotny:2012}. The main purpose of this
  paper is to give another partial answer to this problem. By using
  techniques similar to those of Keesling in \cite{keesling:70a},
  we are able to establish that a Wijsman hyperspace is hereditarily
  normal if and only if it is metrizable.

  \medskip
  The rest of this paper is organized as follows. In Section
  \ref{sec:overview}, an overview on the normality and
  metrizability of basic types of hyperspaces is provided. The
  main result and its proof are given in Section \ref{sec:hnormal}.
  Our terminology and notation are standard. For undefined terms,
  refer to \cite{beer:93}, \cite{burke:84} or \cite{en:89}.

  \section{Normality and metrizability of hyperspaces}
  \label{sec:overview}

  It has been an interesting and challenging problem in general
  topology to characterize normality of the hyperspace of a
  topological space. In
  1955, Ivanova \cite{ivanova:55} showed that $2^{\mathbb N}$
  with the Vietoris topology is not normal, where $\mathbb N$
  is equipped with the discrete topology. Continuing in this
  direction, Keesling \cite{keesling:70a} proved that under
  the CH (Continuum Hypothesis), for a Tychonoff space $X$,
  $(2^X, \tau_V)$ is normal if and only if $(2^X, \tau_V)$
  is compact (and thus if and only if $X$ is compact), where
  $\tau_V$ denotes the Vietoris topology on $2^X$. In
  addition, he also showed in \cite{keesling:70b} that for a
  regular $T_1$ space $X$, a number of covering properties
  of $(2^X,\tau_V)$ (including Lindel\"ofness, paracompactness,
  metacompactness, and meta-Lindel\"ofness) are equivalent to
  compactness of $(2^X,\tau_V)$. Finally, Veli\v{c}ko
  \cite{velicko:75} further showed that Keesling's result on
  normality also holds without the CH. This completely solved
  the normality problem of Vietoris hyperspaces. The normality
  problem of Fell hyperspaces was settled by Hol\'a, Levi and
  Pelant in \cite{hola:99}, where they showed that
  $(2^X, \tau_F)$ is normal if and only if $(2^X,\tau_F)$
  is Lindel\"of, if and only if $X$ is locally compact and
  Lindel\"of, here $\tau_F$ denotes the Fell topology on $2^X$.
  Since in general the Wijsman topology induced by a metric is
  coarser than the Vietoris topology but finer than the Fell
  topology induced by the same metric, the following natural
  question arises.

  \begin{problem} \label{prob:normal}
  Let $(X,d)$ be a metric space. When is the Wijsman hyperspace
  $\left(2^X, \tau_{w(d)}\right)$ a normal space?
  \end{problem}

  Let us temporarily put the normality problem of Wijsman
  hyperspaces aside. Instead, let us recall when a hyperspace
  is metrizable. A classical result claims that for a $T_1$
  space $X$, $(2^X,\tau_V)$ is metrizable if and only if $X$
  is compact and metrizable, refer to \cite[p.298]{en:89}.
  A corresponding result for the Fell topology states that,
  for a Hausforff space $X$, $(2^X, \tau_F)$ is metrizable
  if and only if $X$ is locally compact and second countable
  (and thus Lindel\"of), refer to \cite{flach:64}. For
  Wijsman hyperspaces, we have the following classical result.

  \begin{theorem}[\cite{lechicki-levi:87}] \label{thm:wijmetric}
  Let $(X, d)$ be a metric space. Then $\left(2^X, \tau_{w(d)}
  \right)$ is metrizable  if and only if $(X,d)$ is separable.
  \end{theorem}

  Indeed, if $\{x_n: n \in \mathbb N \}$ is any dense subset
  of $X$, then it can be checked that
  \[
  \varrho_d(A,B) = \sum_{n=1}^\infty \frac{|d(x_n,A)-d(x_n,B)|
  \wedge 1}{2^n}
  \]
  defines a metric on $2^X$ that is compatible with $\tau_{w(d)}$.
  As a consequence of this result, $\left(2^X,\tau_{w(d)}
  \right)$ is Lindel\"of if and only if it is metrizable. Note that
  for any metric space $(X,d)$, we have that $\left(2^X, \tau_{w(d)}
  \right)$ is countably compact if and only if $\left(2^X,\tau_{w(d)}
  \right)$ is compact. Theorem 3.5 in \cite{lechicki-levi:87} also
  claims that $\left(2^X, \tau_{w(d)}\right)$ is metrizable if and
  only if $\{X \}$ is a $G_\delta$-point of $\left(2^X, \tau_{w(d)}
  \right)$. As a consequence, the metrizability of $\left(2^X,\tau_{w(d)}
  \right)$ is equivalent to a large number of generalized metric
  properties. For example, $\left(2^X, \tau_{w(d)} \right)$ is
  metrizable if and only if it has a $G_\delta$-diagonal or it is
  semi-stratifiable.

  \medskip
  In the light of the work of Keesling in \cite{keesling:70b}
  and Veli\v{c}ko in \cite{velicko:75} on the normality of
  Vietoris hyperspaces as well as the work of Hol\'a et al.
  on the normality of Fell hyperspaces, one may wonder
  whether the paracompactness and the metrizability of Wijsman
  hyperspaces are equivalent. These facts motivated Di Maio
  and Meccariello to pose the following natural problem in
  1998, which also brings the normality and the metrizability of
  Wijsman hyperspaces together.

  \begin{problem}[\cite{diMM:98}]\label{prob:dimm}
  It is known that if $(X,d)$ is a separable metric space,
  then $\left(2^X,\tau_{w(d)}\right)$ is metrizable and so
  paracompact and normal. Is the opposite true? Is $\left(
  2^X,\tau_{w(d)}\right)$ normal if, and only if, $\left(2^X,
  \tau_{w(d)}\right)$ is metrizable?
  \end{problem}

  Note that neither Problem \ref{prob:normal} nor Problem \ref{prob:dimm}
  is completely solved. An affirmative solution to Problem
  \ref{prob:dimm} would also solve Problem \ref{prob:normal}.
  The following partial answer to Problem \ref{prob:dimm} was recently
  established by Hol\'{a} and Novotn\'{y} in \cite{hola-novotny:2012}.

  \begin{theorem}[\cite{hola-novotny:2012}] \label{thm:hola}
  Let $(X, \|\cdot\|)$ be a normed linear space, and let $d$ be the
  metric on $X$ induced by the norm $\|\cdot \|$. Then $\left(2^X,
  \tau_{w(d)}\right)$ is normal if and only if it is metrizable.
  \end{theorem}

  Given a topological space $X$, define ${\rm nlc}(X)$ by
  \[
  {\rm nlc}(X) = \{x \in X:\ x \mbox{ has no compact neighbourhood
  in } X\}.
  \]
  The following result was established by Chaber and Pol in
  \cite{chaber-pol:02}.

  \begin{theorem}[\cite{chaber-pol:02}] \label{thm:chaber}
  If $(X,d)$ is a metric space such that ${\rm nlc}(X)$ is non-separable,
  then $\left(2^X, \tau_{w(d)}\right)$ contains a closed copy of
  ${\mathbb N}^{\omega_1}$.
  \end{theorem}

  Since ${\mathbb N}^{\omega_1}$ is non-normal, as a corollary of
  Theorem \ref{thm:chaber}, if $(X,d)$ is a metric space such that
  ${\rm nlc}(X)$ is non-separable, then $\left(2^X, \tau_{w(d)}\right)$
  is non-normal. In particular, if $(X, \|\cdot\|)$ is a non-separable
  normed linear space and $d$ is the metric on $X$ induced by the
  norm $\|\cdot \|$, then $\left(2^X,\tau_{w(d)}\right)$ is non-normal.
  Therefore, Theorem \ref{thm:hola} can be viewed as a corollary of
  Theorem \ref{thm:chaber}.

  \section{The main result}
  \label{sec:hnormal}

  In this section, we shall prove the following main result of this paper,
  which can be treated as a partial answer to Problem \ref{prob:normal}
  and Problem \ref{prob:dimm}.

  \begin{theorem}\label{thm:hnormal}
  Let $(X,d)$ be a metric space. The following
  are equivalent.
  \begin{itemize}
  \item[(i)] $\left(2^X, \tau_{w(d)}\right)$ is metrizable.
  \item[(ii)] $\left(2^X\setminus \{X\}, \tau_{w(d)}\right)$
  is paracompact.
  \item[(iii)] $\left(2^X, \tau_{w(d)}\right)$  is hereditarily normal.
  \end{itemize}
  \end{theorem}

  To prove the above theorem, we use the embedding techniques,
  similar to those used by Keesling in \cite{keesling:70a}. In what
  follows, the ordinals $\omega_1$ and $\omega_1+1$ are viewed as
  topological spaces equipped with the order topology.

  \begin{proposition}\label{prop:normal}
  Let $(X,d)$ be a non-separable metric space. Then for any $n\ge 1$,
  the Wijsman hyperspace $\left(2^X, \tau_{w(d)}\right)$ contains a
  copy of $(\omega_1+1)^n$.
  \end{proposition}

  \begin{proof}
  Let $Y_n=(\omega_1+1)^n$. Since $(X,d)$ is non-separable, there
  exist $\varepsilon>0$ and a set $D\subseteq X$ with $|D| =
  \aleph_1$ which is $\varepsilon$-discrete, that is,
  $d(x, y)\ge \varepsilon$ for all distinct $x, y \in D$.
  Let $n \ge 1$. We express $D$ as the disjoint
  union $D=\bigcup_{i=0}^n D_i$ such that $D_0 =\{d\}$
  and $|D_i|=\aleph_1$ for all $1\le i \le n$. Next, for
  $1\le i \le n$, we enumerate $D_i$ as $D_i=\{x_\alpha^i:
  \alpha <\omega_1\}$, and for each $\alpha \le
  \omega_1$, we put $L_\alpha^i = \{x_\lambda^i \in D_i:
  \lambda <\alpha\}$. Obviously, each $L_\alpha^i$ is closed
  in $X$. Define a mapping
  $\varphi: Y_n \to \left(2^X,\tau_{w(d)}\right)$ by the formula
  $\varphi(\alpha_i)= D_0 \cup\bigcup_{i=1}^n
  L_{\alpha_i}^i$. 
  It is clear that $\varphi$ is one-to-one.

  \medskip
  To see that $\varphi$ is continuous, suppose
  $\varphi(\alpha_i) \cap V \ne \emptyset$
  for some open set $V\subseteq X$. If $D_0 \cap V\ne
  \emptyset$, there is nothing to verify. So, we assume
  that $L_{\alpha_i}^i \cap V\ne \emptyset$ for some $1\le
  i\le n$. Hence, there is some non-limit ordinal $\lambda_i
  <\alpha_i$ such that $L_{\lambda_i}^i \cap V \ne \emptyset$.
  It follows that for any neighborhood $N(\alpha_j)$ of
  $\alpha_j$ with $j\ne i$, we have
  \[
  \varphi\left(\prod_{j<i} N(\alpha_j) \times
  (\lambda_i, \alpha_i] \times \prod_{j>i} N(\alpha_j)
  \right) \subseteq V^-.
  \]
  On the other hand, if
  $d(x,\varphi(\alpha_i))>r$ for some
  $x\in X$ and $r>0$, then for any $\lambda_i \le \alpha_i$
  we have $d(x,\varphi(\lambda_i)) >r$.
  Thus, we have verified that $\varphi$ is continuous at
  any point $(\alpha_i) \in Y_n$.

  \medskip
  Since $Y_n$ is compact, the continuous one-to-one mapping
  $\varphi$ is an embedding.
  \end{proof}

  \begin{corollary}
  Let $(X,d)$ be a metric space. Then for the space
  $\left(2^X, \tau_{w(d)}\right)$, metrizability is equivalent
  to each one of the following properties: Fr\' echetness,
  sequentiality, countable tightness.
  \end{corollary}

  \begin{proof}
  Proposition~\ref{prop:normal} shows, in particular, that if
  $(X,d)$ is non-separable, then $\left(2^X,\tau_{w(d)} \right)$
  contains a copy of $\omega_1+1$. Since $\omega_1+1$ does not
  have countable tightness, the conclusion follows.
  \end{proof}

  \begin{question}
  Let $(X,d)$ be a non-separable metric space. Can $\left(2^X,
  \tau_{w(d)} \right)$ contain a copy of $(\omega_1+1)^{\omega}$ or $(\omega_1+1)^{\omega_1}$?
  \end{question}

  For the proof of our next proposition, we need an
  auxiliary result.

  \begin{lemma}\label{lem:cluster}
  Let $(X,d)$ be a metric space. If ${\mathcal F}$ is a
  directed family in $2^X$ such that $H=\bigcup {\mathcal F}$
  is closed, then $H$ is a limit point of the net
  $(\mathcal F, \subseteq)$ in $\left(2^X, \tau_{w(d)}\right)$.
  \end{lemma}

  \begin{proof}
  The proof of this lemma is straightforward, and thus is
  omitted.
  \end{proof}

  In his proof of the equivalence of normality and compactness
  for Vietoris hyperspaces (under the CH), Keesling
  \cite{keesling:70a} established that, for a noncompact
  Tychonoff space $X$, $(2^X, \tau_V)$ contains a
  closed copy of the space $\omega_1\times (\omega_1+1)$.
  We are not able to obtain a similar embedding in the
  Wijsman hyperspace of a non-separable metric space $(X,d)$.
  However, we have the following result.

  \begin{proposition} \label{prop:subspace}
  Let $(X,d)$ be a non-separable metric space. Then the
  subspace $2^X\setminus \{X\}$ of $\left(2^X,\tau_{w(d)}\right)$
  contains a closed copy of the space $\omega_1\times (\omega_1+1)$.
  \end{proposition}

  \begin{proof}
  Since $(X,d)$ is non-separable,
  there exist $\varepsilon>0$ and an $\varepsilon$-discrete
  proper subset $D=\{x_{\alpha,\beta}: \alpha<\omega_1
  \mbox{ and } \beta\leq\omega_1\}$ of $X$, with $x_{\alpha,
  \beta}\ne x_{\alpha',\beta'}$ for $(\alpha,\beta)\ne
  (\alpha',\beta')$.
  For every $\alpha<\omega_1$, let $D_\alpha=\{x_{\alpha,
  \beta}: \beta\leq\omega_1\}$ and $G_\alpha=S_d(D_\alpha,
  \frac{\varepsilon}{4})$. For every $\alpha<\omega_1$ and
  each $\beta\leq\omega_1$, let
  \[
  F_\alpha=X\setminus\left(\bigcup_{\gamma\geq\alpha}G_\gamma
  \right)
  \]
  and
  \[
  S_\beta=\{x_{\gamma,\delta}: \gamma<\omega_1 \mbox{ and }
  \delta<\beta\}.
  \]
  Note that the families ${\mathcal F}=\{F_\alpha: \alpha<
  \omega_1\}$ and $\mathcal S =\{S_\beta: \beta\leq\omega_1\}$
  are ``continuously increasing", in the sense that
  $F_\alpha=\bigcup \{ F_{\gamma+1}:\gamma<\alpha\}$ and
  $S_\beta=\bigcup \{ S_{\delta+1}: \delta<\beta\}$ for
  all $0<\alpha<\omega_1$ and $0<\beta\leq\omega_1$.
  We show that the subspace
  \[
  {\mathcal H}=\{F_\alpha\cup S_\beta: \alpha<\omega_1
  \mbox{ and }\beta\leq\omega_1\}
  \]
  of $(2^{X},\tau_{w(d)})$
  is homeomorphic to the product space $\omega_1\times
  (\omega_1+1)$. Define a mapping $\varphi: \omega_1\times
  (\omega_1+1)\to {\mathcal H}$ by the formula $\varphi
  (\alpha,\beta)=F_\alpha\cup S_\beta$, and note that
  $\varphi$ is one-to-one and onto.

  \medskip
  To show that $\varphi$ is continuous, let $A\subseteq
  \omega_1\times(\omega_1+1)$, and let
  $(\alpha,\beta)\in\overline{A}$. We show that
  $\varphi(\alpha,\beta)\in\overline{\varphi(A)}$.
  Let $A'=\{(\gamma,\delta)\in A: \gamma\leq\alpha\
  {\rm and}\ \delta\leq\beta\}$, and note that
  $(\alpha,\beta)\in\overline{A'}$. Note that, for all
  $(\gamma,\delta),(\gamma',\delta')\in A'$,
  there exists $(\mu,\nu)\in A'$ such that $\mu\geq
  \max (\gamma,\gamma')$ and
  $\nu\geq\max(\delta,\delta')$. As a consequence,
  the family $\{F_\gamma\cup S_\delta:
  (\gamma,\delta)\in A'\}$ is directed. Since
  $(\alpha,\beta)\in\overline{A'}$, we have for all
  $\alpha'<\alpha$ and $\beta'<\beta$ that there exists
  $(\gamma,\delta)\in A'$ such that
  $\gamma\geq\alpha'$ and $\delta\geq\beta'$. As $\mathcal F$
  and $\mathcal S$ are continuously increasing,
  it follows that
  \[
  \bigcup\{F_\gamma\cup S_\delta:
  (\gamma,\delta)\in A'\}=F_\alpha\cup S_\beta.
  \]
  From the foregoing it follows by Lemma \ref{lem:cluster}
  that the net $(\{F_\gamma\cup S_\delta:
  (\gamma,\delta)\in A'\},\subseteq)$ converges to
  $F_\alpha \cup S_\beta$ in $\tau_{w(d)}$.
  As a consequence, $\varphi(\alpha,\beta)\in
  \overline{\varphi(A')}\subseteq\overline{\varphi(A)}$.
  We have shown that $\varphi$ is continuous.

  \medskip
  Next, we show that $\varphi$ is open. Let $W$ be an open
  subset of $\omega_1\times(\omega_1+1)$. To show that
  $\varphi(W)$ is open in $\mathcal H$, let $(\alpha,\beta)
  \in W$. Denote by $J$ the element $\varphi(\alpha,\beta)=
  F_\alpha\cup S_\beta$ of the set $\varphi(W)$. There exist
  $\gamma<\alpha$ and $\delta<\beta$ such that $(\gamma,\alpha]
  \times (\delta,\beta] \subseteq W$. Let $E=\{x_{\alpha,
  \beta}, x_{\alpha,\delta}, x_{\gamma,\beta}\}$ and
  \[
  {\mathcal N}_{J,E,\varepsilon/2}=\left\{H\in\mathcal H:
  |d(x,H)-d(x,J)|
  <\frac{\varepsilon}{2} \mbox{ for every }x\in E\right\}.
  \]
  Note that ${\mathcal N}_{J,E,\varepsilon/2}$ is a neighborhood
  of $J$ in $\mathcal H$. We show that ${\mathcal N}_{J,E,
  \varepsilon/2} \subseteq\varphi(W)$.
  Let $H\in {\mathcal N}_{J,E,\varepsilon/2}$, and let
  $\mu<\omega_1$ and $\nu\leq
  \omega_1$ be such that $H=F_\mu \cup S_\nu$. To show that
  $H\in\varphi(W)$, we need to show that the inequalities
  $\gamma<\mu\leq\alpha$ and $\delta<\nu \leq\beta$ hold.
  For the element $x_{\alpha,\beta}$ of $E$, we have
  $x_{\alpha,\beta}\in G_\alpha\subseteq X\setminus F_\alpha$
  and $x_{\alpha,\beta}\not\in S_{\beta}$. It follows that
  $x_{\alpha,\beta}\not\in J$ and hence that
  $d(x_{\alpha,\beta},J)\geq\varepsilon$. As a consequence,
  \[
  d(x_{\alpha,\beta},H)\geq
  d(x_{\alpha,\beta},J)-|d(x_{\alpha,\beta},H)-d(x_{\alpha,
  \beta},J)|\geq \varepsilon-\frac{\varepsilon}{2}>0.
  \]
  By the foregoing, we have that $x_{\alpha,\beta}\not\in H$,
  and this means that $x_{\alpha,\beta}\not\in F_\mu$
  and $x_{\alpha,\beta}\not\in S_\nu$. It follows that we
  have $\mu\leq\alpha$ and $\nu\leq\beta$.
  For the element $x_{\alpha,\delta}$ of $E$, we have
  $x_{\alpha,\delta}\in S_{\beta}\subseteq J$, and hence
  $d(x_{\alpha,\delta},J)=0$. It follows that
  \[
  d(x_{\alpha,\delta},H)=|d(x_{\alpha,\delta},H)-
  d(x_{\alpha,\delta}, J)|<\frac{\varepsilon}{2}.
  \]
  Since $H\subseteq D$ and $x_{\alpha,\delta}\in D$, it
  follows further, by $\varepsilon$-discreteness of $D$,
  that $x_{\alpha,\delta}\in H$. Since $H=F_\mu\cup S_\nu$,
  we have either $x_{\alpha,\delta}\in F_\mu$ or $x_{\alpha,
  \delta}\in S_\nu$. In the first case, since $x_{\alpha,
  \delta}\in D_\alpha\subseteq G_\alpha$, we would have that
  $\alpha<\mu$; however, we showed above that $\mu\leq\alpha$.
  Hence we must have that $x_{\alpha,\delta}\in S_\nu$. It follows
  that $\delta<\nu$. We have shown that $\delta<\nu \leq\beta$.
  Similarly, for the element $x_{\gamma,\beta}$ of $E$, we
  have that $x_{\gamma,\beta}\in F_{\alpha}\subseteq J$, and hence
  that $d(x_{\gamma,\beta},J)=0$. It follows that $d(x_{\gamma,
  \beta},H)<\frac{\varepsilon}{2}$, and further, that
  $x_{\gamma,\beta}\in H$. As a consequence, we have either
  $x_{\gamma,\beta}\in F_\mu$ or $x_{\gamma,\beta}\in S_\nu$.
  In the second case we would have that $\beta<\nu$, but this does not
  hold, since we showed above that $\nu\leq\beta$. Hence we
  must have $x_{\gamma,\beta}\in F_\mu$, and it follows from this that
  $\gamma<\mu$. We have shown that $\gamma<\mu\leq\alpha$.
  This completes the proof of openness of $\varphi$.

  \medskip
  We have shown that the subspace $\mathcal H$ of $\left(2^{X},
  \tau_{w(d)}\right)$ is homeomorphic to the space $\omega_1
  \times(\omega_1+1)$. Note that $\bigcup\mathcal H=D$. As a
  consequence, $X\not\in\mathcal H$. To complete the proof,
  we show that $\mathcal H$ is closed in the subspace $2^{X}
  \setminus\{X\}$ of $\left(2^{X},\tau_{w(d)}\right)$. Let
  $K\in\overline{\mathcal H}\setminus\mathcal H$. To show
  that $K=X$, assume on the contrary that $X\setminus K\ne
  \emptyset$. Let $y\in X\setminus K$. There exists $\alpha_0
  <\omega_1$ such that $y\not\in\bigcup_{\gamma> \alpha_0}
  G_\gamma$. Note that $y\in F_\alpha$ for each $\alpha>
  \alpha_0$. The subset $\mathcal H_0=\{F_\alpha\cup S_\beta:
  \alpha\leq\alpha_0 \mbox{ and }\beta\leq\omega_1\}$ of
  $\mathcal H$ is compact, because it is homeomorphic to
  $[0,\alpha_0]\times [0,\omega_1]$. Since $K\in
  \overline{\mathcal H}\setminus\mathcal H$, it follows
  that $K\in\overline{\mathcal H\setminus\mathcal H_0}$. Let
  $r=d(y,K)$ and consider the neighborhood
  \[
  {\mathcal M}_{K,\{y\}, r}=\{H\in 2^{X}: |d(y,H)-d(y,K)|<r\}
  \]
  of $K$ in $\left(2^{X},
  \tau_{w(d)}\right)$. It follows from the foregoing, that
  there exist $\alpha>\alpha_0$ and $\beta\leq\omega_1$ such
  that $F_\alpha\cup S_\beta\in {\mathcal M}_{K,\{y\}, r}$.
  However, now we have that $y\in F_\alpha$ and hence $d(y,
  F_\alpha\cup S_\beta)=0$. Since $F_\alpha \cup S_\beta\in
  {\mathcal M}_{K,\{y\}, r}$, we have $d(y,K)<r\,$; this, however,
  contradicts with the definition of $r$. We have shown that
  $\overline{\mathcal H}\setminus\mathcal H\subseteq\{X\}$
  and hence that $\mathcal H$ is closed in $2^{X}\setminus
  \{X\}$.
  \end{proof}

  \begin{corollary}\label{coro:hnormal}
  Let $(X,d)$ be a metric space. The following
  are equivalent.
  \begin{itemize}
  \item[(i)] $\left(2^X, \tau_{w(d)}\right)$ is metrizable.
  \item[(ii)] $\left(2^X\setminus \{X\}, \tau_{w(d)}\right)$
  is metacompact.
  \item[(iii)] $\left(2^X\setminus \{X\}, \tau_{w(d)}\right)$ is
  meta-Lindel\"{o}f.
  \item[(iv)] $\left(2^X\setminus \{X\}, \tau_{w(d)}\right)$ is
  orthocompact.
  \end{itemize}
  \end{corollary}

  \begin{proof}
  We only need to show that (iv) $\Rightarrow$ (i). Assume
  that $(X,d)$ is not separable. By
  Proposition \ref{prop:subspace}, $\left(2^X
  \setminus \{X\}, \tau_{w(d)}\right)$ contains a closed copy
  of $\omega_1 \times (\omega_1+1)$. As $\left(2^X
  \setminus \{X\},\tau_{w(d)}\right)$ is orthocompact, then
  $\omega_1 \times (\omega_1+1)$ is orthocompact, which
  contradicts with a result of Scott in \cite{scott:75}.
  \end{proof}

  We now use Proposition \ref{prop:subspace} to prove Theorem
  \ref{thm:hnormal}.

  \begin{proof}[Proof of Theorem \ref{thm:hnormal}]
  We only need to prove that (iii) implies (i). Assume that
  (iii) holds, but $\left(2^X, \tau_{w(d)}\right)$ is not
  metrizable. By Theorem \ref{thm:wijmetric}, $(X,d)$
  is not separable. By Proposition ~\ref{prop:subspace},
  $\left(2^X\setminus \{X\}, \tau_{w(d)}\right)$ contains a
  closed copy of $\omega_1\times (\omega_1+1)$. Since
  $\omega_1\times (\omega_1+1)$ is not normal, (iii) does
  not hold. This is a contradiction.
  \end{proof}

  We conclude this paper with the following open question.

  \begin{question}
  Let $(X,d)$ be a metric space. If $\left(2^X,\tau_{w(d)}\right)$
  is non-normal, does $\left(2^X,\tau_{w(d)}\right)$ contain a
  closed copy of $\omega_1 \times (\omega_1+1)$?
  \end{question}

  Note that there exists a metric space $(X,d)$ such that $\left(2^X,\tau_{w(d)}\right)$ is non-normal, but $\left(2^X,
  \tau_{w(d)}\right)$ contains no closed copy of ${\mathbb
  N}^{\omega_1}$. Indeed, take any set $X$ with $|X|=\omega_1$
  and equip $X$ with the 0-1 metric $d$. By Remark 3.1 of
  \cite{chaber-pol:02}, $\left(2^X,\tau_{w(d)}\right)$ is
  homeomorphic to $\{0,1\}^{\omega_1} \setminus \{\bf 0\}$,
  which is locally compact. Thus, $\left(2^X,\tau_{w(d)}\right)$
  contains no closed copy of ${\mathbb N}^{\omega_1}$.

   \end{document}